\documentclass[twoside]{amsart}
\usepackage{amsmath,amssymb,amscd,amsthm,amsxtra,amsfonts}
\usepackage{mathtools,mathrsfs,dsfont,xparse}
\usepackage{graphicx,epstopdf}
\usepackage{multirow}
\usepackage{cases}
\usepackage{enumerate}
\usepackage{xcolor}
\usepackage{tikz,pgfplots}
\usetikzlibrary{matrix}
\usepackage{mathtools}
\usepackage[margin=1 in]{geometry}
\usepackage{todonotes}
\usepackage{listings}
\usepackage{tcolorbox}
\usepackage{cancel}

\usepackage{color}
\usepackage{circledsteps,}

\allowdisplaybreaks[4]

\makeatletter
\newcommand{\dotDelta}{{\vphantom{\Delta}\mathpalette\d@tD@lta\relax}}
\newcommand{\d@tD@lta}[2]{%
  \ooalign{\hidewidth$\m@th#1\mkern-1mu\cdot$\hidewidth\cr$\m@th#1\Delta$\cr}%
}
\makeatother

\date{}

\usepackage{multicol,bbm}
\usepackage[colorlinks=true, pdfstartview=FitV, linkcolor=blue, citecolor=blue, urlcolor=blue]{hyperref}

\numberwithin{equation}{section}

\usepackage{cjhebrew}

\DeclareMathOperator{\re}{Re}

\usepackage[numbers,sort]{natbib}

{\theoremstyle {definition} \newtheorem {defn} {Definition} [section] }
{\theoremstyle {plain}  \newtheorem {thm} [defn] {Theorem}}
{\theoremstyle {plain}  }
{\theoremstyle {plain} \newtheorem {prop} [defn]{Proposition}}
{\theoremstyle {plain} \newtheorem {lem}[defn] {Lemma}}
{\theoremstyle {definition} \newtheorem {rmk}[defn] {Remark}}
{\theoremstyle {plain} }

\def\R{{\mathbb{R}}}
\def\C{{\mathbb{C}}}

\def\Z{{\mathbb{Z}}}

\def\e{{\varepsilon}}

\newcommand{\norm}[1]{\left\|#1\right\|}

\title{An inverse Problem for the cubic $\alpha$-NLS in Sobolev spaces}
\author{Zachary Lee, Nata\v{s}a Pavlovi\'{c}
}
\date{\today}

\begin{document}

\maketitle

\begin{abstract} 
In this work, we address an inverse problem for a defocusing cubic nonlinear Schr\"{o}dinger (NLS) equation in dimensions $d\in\{1, 2,3\}$ in a range of Sobolev spaces $H^s(\R^d)$ by employing the method of approximate solutions. 
We recover a smooth, space-dependent and compactly supported function $\alpha$ that controls the nonlinearity (and thus self-interaction strength) in a multiplicative fashion. 
To the best of our knowledge, this is the first work based on approximate solutions in Sobolev spaces that treats an inverse problem for the NLS and provides explicit recovery of $\alpha$.

\end{abstract}

\setcounter{tocdepth}{1}
\tableofcontents

\parindent = 10pt     
\parskip = 8pt

\section{Introduction}

% {\color{blue}{
% To do list:
% \begin{enumerate}
%     \item Check reference 13
%     \item Check whether papers on NLS using inverse scattering map give recovery
%     \item See if paragraphs after theorem 2.1 talk about the challenge, solution to the challenge and difference from previous works 
%     \item 
% \end{enumerate}

% }}

% {\color{blue}What we do} 

\par 
The forward problem, also called initial-value problem, associated with the nonlinear Schrödinger (NLS) equation for $u:\R^d \times \R \to \C$ 
\begin{align}\label{generic NLS}
        i\partial_t u + \Delta u = \mu \, \alpha(x) \,|u|^{p-1}u
\end{align}
and the initial data  $u(t_0,x)=u_{t_0}(x)$, where $\mu=\pm 1, p>1$, $\alpha\equiv 1$ has been extensively studied, see e.g. \cite{KPV, Bourgain1, GV, KM, I-Team, Bourgain2, Cazenave, tao2006nonlinear}. Such a problem asks the question of constructing a solution to the NLS equation given all the coefficients and initial data. On the other hand, an inverse problem asks the question of recovering an unknown function or parameter in the equation, such as $\alpha(x)$ in \eqref{generic NLS}, given knowledge of the solution. The inverse-scattering problem for the NLS equation, in which knowledge of the scattering map is used to solve for a coefficient or parameter in the equation, is an example of an inverse problem that has received a good deal of attention, see e.g. \cite{Weder01, Weder2001PMAS, WederMMAS, Wave_Paper, Watanabe}. Another type of inverse problem, and the one this work is concerned with, attempts to compute an unknown function $\alpha(x)$ in \eqref{generic NLS} by relying on the existence of the solution to the initial value problem and by constructing an explicit \textit{approximate solution} of the equation valid in a certain regime. The approximate solution should carry, usually hidden in it, information about the desired function. Then, one starts with an initial data in the above mentioned regime, and if the actual solution of the equation can be shown to stay in that regime, one can hope to recover the unknown function from measurement of the true solution (which should be close to the approximate solution). This final step of the process is recovery of the desired function $\alpha(x)$ from the approximate solution. 

% {\color{red} In future, see if we need to make any changes to emphasize that the approximate solution is needed for recovery}

\par 
In this work, we address the inverse problem for $\eqref{generic NLS}$ with $p=3$, $\mu=+1$ in dimensions $d\in\{1, 2,3\}$ in a range of Sobolev spaces $H^s(\R^d)$ by employing the method of approximate solutions. We focus on these dimensions as they are the most physically relevant. Also, going to higher dimensions would lead to a more complicated global theory for the equation, which would distract us from the main purpose of this paper. We also recover a smooth, space-dependent function $\alpha$ that controls the nonlinearity (and thus self-interaction strength) in a multiplicative fashion. Before we present details of our work, let us review previous works that are relevant to the paper at hand. 

\begin{itemize} 

\item The study of inverse problems for the NLS based on the scattering map began in earnest in 1974, when Strauss \cite{Strauss} showed that $\alpha$ and $p$ in \eqref{generic NLS} could be reconstructed from the small-data scattering map if $\alpha$ was real-valued, smooth and bounded, and $p$ satisfied some given conditions. Such a result was later generalized to NLS equations with a potential and a more generic nonlinearity. In particular, Weder \cite{Weder97, Weder01} showed that an additional linear potential term $V_0(x)$ could be recovered via the small-amplitude limit of the scattering map. Subsequently, in \cite{Weder2001PMAS, WederMMAS} Weder also proved that the same approach yields the recovery of a potential $V_0$ and a nonlinearity of the form $\sum_{j\ge 1} V_j(x) |u|^{2j+2j_0}u$, i.e. $V_j$ for $j=0, 1, \ldots$ and $j_0$. In \cite{Watanabe}, Watanabe proved scattering in three-dimension for \eqref{generic NLS} in the range $4/3<p<4$ for certain decaying coefficients $\alpha$ and showed how to recover $\alpha$ from the large-data scattering map.  Subsequently, in \cite{Mur}, Murphy lowered the regularity requirements for the recovery identified in \cite{Strauss, Weder2001PMAS, WederMMAS} and extended the results of \cite{Watanabe} to the mass-critical and mass-subcritical regime in $d\ge 3$. Regarding an NLS with a generalized nonlinearity, we note that in the work \cite{KRMPV}, Kilip, Murphy and Vișan showed that for certain generalized nonlinearities such that the corresponding NLS admits small-data scattering, the scattering map uniquely determines the nonlinearity.

 \item On the other hand, recently the inverse problem for the NLS has been also studied using the method of approximate solutions. In particular, in \cite{HOGAN2023127016}, Hogan, Grow and Murphy studied the inverse problem for the NLS \eqref{generic NLS} with $p=3$ and $d=2,3$ in the Wiener algebra $\mathcal{F}L^1_{x}(\R^d)$. Their approximate solution has a small amplitude $\e^q$ for $q>1$, a frequency that is independent of $\e$, and an amplitude function living on a large ball of radius of the order $1/\e$. In \cite{Lee-Yu}, the first author of this paper and Yu  extended the results of \cite{HOGAN2023127016} to generalized Schrödinger equations with higher order dispersion terms and real-analytic nonlinearities. In order to compare the state of the results for NLS with results for the nonlinear wave equation, we would like to mention that in the context of the nonlinear cubic wave equation in dimensions $d=2,3$, Sá Barreto and Stefanov in \cite{Wave_Paper} considered the inverse problem for $\alpha(x)$ in Sobolev spaces. Thanks to the finite speed of propagation property of the wave equation, the work \cite{Wave_Paper}, considers initial data only lying in $H^{d/2+\delta }_{\mathrm{loc}}(\R^d), \delta>0$ as a result of a lack of decay at infinity.  Furthermore, the authors of \cite{Wave_Paper} explicitly recovered the function $\alpha$. 

\end{itemize}
\par 
To the best of our knowledge, there are no previous works on the inverse problem for the NLS \eqref{generic NLS} in Sobolev spaces based on approximate solutions. In addition, recovery results of the type mentioned above for the nonlinear wave equation \cite{Wave_Paper} are not known in the context of the inverse problem for the NLS via approximate solutions. The paper at hand exactly addresses that gap.
\par
More precisely, we obtain the following results.

\begin{thm}\label{Main Theorem}
Let $d=1,2,3$, $\alpha\in C_c^\infty(\R^d)$, $\alpha \ge 0 ,T_0>0$ such that  $\operatorname{supp}\alpha \subseteq B(0,T_0)$ and $T>\frac{T_0}{2}$. Also let $\xi \in \R^d$ such that $|\xi|=1$ and let $h$
% $h\in \frac{3T{\color{blue}+x_0\cdot \xi}}{2\pi \N}$
be such that $0<h<1$. Suppose $\psi\in C_c^\infty (\R^d)$ and $\operatorname{supp} \psi \subseteq B(0,R)$ and  $\psi(x)\neq 0$ on $B(0,T_0)$. Define
\begin{align} \label{th-st-v}
    v(t,x)= h^{-1/2} e^{i(x\cdot \xi/h - t|\xi|^2/h^2)} \psi(x  -2 \xi t/h -2\xi T) \exp \left(-i |\psi(x - 2\xi t/h - 2\xi T)|^2 \int_{-T}^{t/h} \alpha(x-2\xi t/h +2\xi s)\,ds\right).
\end{align}
If $u\in C_t^0 H^{d/2+\delta}_x([-Th, Th]\times \R^d)$, with $\delta(d)>0$ for $d\in\{1,2\}$ and $\delta(d)\ge 1/2$ for $d=1$, is the unique, strong solution (in the sense of Definition \ref{Strong solution}) to the following initial value problem
\begin{align}
   \begin{cases} \label{alpha NLS}
   i\partial_t u + \Delta u=\alpha(x) u|u|^2 \\
    u(-Th,x) = h^{-1/2} e^{i\frac{x\cdot \xi}{h} + i\frac{T|\xi|^2}{h}} \psi(x)
   \end{cases}
\end{align}
then
\begin{align*}
    \norm{u-v}_{C_{t,x}^0 ([-Th, Th]\times\R^d)} \lesssim_{T,d, \delta} h^{1/2}.
\end{align*}
In particular, for every $x_0\in\R^d$,
\begin{align}\label{u evaluated}
    u(Th, 4\xi T+x_0) = h^{-1/2}  e^{i(3T+x_0\cdot \xi) /h }e^{-i\psi(x_0)^2 X\alpha(x_0, \xi) }\psi(x_0) + O(h^{1/2}),
\end{align}
where
\begin{align}\label{X(alpha) equation}
    X\alpha(x_0, \xi)=  \frac{1}{2}\int_{0}^\infty  \alpha(x_0+s\xi)\,ds.
\end{align}
\end{thm}

\begin{rmk}
$X\alpha(x_0,\xi)$, defined in \eqref{X(alpha) equation}, is closely related to the classical X-ray transform $\mathbf{P}\alpha(x,\xi)$  as defined in \cite{Natterer}. They are related by
\begin{align}
    \mathbf{P}\alpha(x_0,\xi):= \int_{-\infty}^\infty \alpha(x_0 + s\xi)\,ds = 2\left(X\alpha(x_0, \xi) + X\alpha(x_0, -\xi)\right).
\end{align}
\end{rmk}
% \begin{rmk}
%     The requirement that $h\in \frac{3T}{2\pi \N}$ (that is $h$ is \textit{quantized}) is novel to our work and is a vital ingredient for our recovery process.
% \end{rmk}
In this paper we also prove the following result that recovers $X\alpha$. 

\begin{thm}\label{Recovery Theorem}
    Assume the same hypotheses as in Theorem \ref{Main Theorem}. Also, assume that $\psi(x)$ is constant with non-zero value $K$ in the non-empty, bounded region
    \begin{align}
            C_{\xi, T_0} = \left(\{x_0\in \R^d: X\alpha(x_0,\xi)=0\} \cap B(0,2T_0)\right) \cup \operatorname{supp}\alpha
    \end{align}
    Then, there exists an 
    integer-valued $g$ such that $X(\alpha)$, defined in equation \eqref{X(alpha) equation}, may be recovered from 
    \begin{align*}
    K^2 X\alpha(x_0, \xi) = -\operatorname{Im} \log \left\{K^{-1}e^{-i(3T+x_0\cdot \xi) /h }h^{1/2} u(Th, 4\xi T+x_0)\right\} + 2\pi g(x_0, \xi) + O(h).
    \end{align*}
\end{thm}
\begin{rmk}
% We note that the above theorem provides explicit recovery of $X(\alpha)$ by taking $K_{x_0}$ even smaller if necessary such that $O(K_{x_0}) \ll X(\alpha)$ (for example, this is guaranteed if $\norm{\psi}_{L^\infty_x(\R^d)}$ is small enough). Then, we pick $h$ small enough such that $O(K_{x_0}^{-3}h) \ll X(\alpha)$. 
This in turn recovers $\alpha$ itself for $d\in\{2,3\}$ via the X-ray inverse transform by varying $\xi\in S^{d-1}(\R^d)$ via the inversion formula given in \cite{Natterer}.

% since the $O(\cdot)$ error terms are much smaller than the first term on the right-hand side, which is $O(1)$. 

\end{rmk}
\begin{rmk}
    For $d=1$, we may take  $\xi=1$ in \eqref{X(alpha) equation} and note that 
    \begin{align}
        X\alpha(x_0, 1) &= \frac{1}{2}\int_0^\infty \alpha(x_0+ s)\,ds.
    \end{align}
    Then, 
    \begin{align}
        \frac{d}{dx_0} X\alpha(x_0, 1) &= \frac{1}{2} \int_0^\infty \alpha'(x_0+s)\,ds \\
        &= -\frac{1}{2}\alpha(x_0).
    \end{align}
Hence, we recover $\alpha$ in $d=1$ given knowledge of the tangent of $X\alpha(x_0, 1)$ in $x_0$, which we may compute given knowledge of $y\mapsto X\alpha(y, 1)$ in a neighborhood of $x_0$.
\end{rmk}
% \begin{rmk}
%     Note that we do not consider the case $d=1$ since in this case $\mathbf{P}\alpha(x_0,\xi)$ reduces to 
%     \begin{align}
%         \int_{-\infty}^\infty \alpha(x_0+ \xi s)\,ds = \frac{1}{|\xi|}\int_{\R} \alpha(s)\,ds,
%     \end{align}
%     which is not enough information to recover $\alpha$. Indeed, any two functions whose integrals are equal will have the same X-ray transform.
% \end{rmk}
\par

\par 

Our approach to proving the above result is inspired by the work on the nonlinear wave equation \cite{Wave_Paper} and is based on constructing an approximate solution for a large-amplitude and large-frequency initial wave packet. Both \cite{Wave_Paper} and our work use an ansatz in the form of a truncated asymptotic expansion in powers of $h$, an approach that also appears in \cite{rauch2015hyperbolic}. One important advantage of working with large data is that the signal is more likely to rise above any background noise than if using small amplitude data. While doing so, we face two challenges that we describe below.

\begin{itemize}

\item The first challenge is that the $\alpha$-NLS \eqref{alpha NLS} has an infinite speed of propagation, unlike the nonlinear wave equation, and so we must develop our well-posedness theory on the entire space $\R^d$. This is in contrast to how the work \cite{Wave_Paper} treats the problem; the authors in fact use the finite speed of propagation of the wave equation and work in a large ball of finite radius $R_1>0$. As a result, they can consider initial data of the form $\psi(x\cdot \xi)$ for $\psi\in C_c^\infty(\R^d)$, which is only locally integrable because of a lack of decay at infinity in dimension $d\ge 2$. On the other hand, we consider initial data of the form $\psi(x) \in C_c^\infty(\R^d)$ (note the lack of inner product) which lies in the full Sobolev space $H^{d/2+\delta}_x(\R^d)$.

% {\color{red}but we show that it may still hold if we restrict the values of $h$ to cancel out a rapidly varying phase in our approximate solution, which otherwise would make it impossible to recover the comparably small signal from $\alpha$. This has the implication that the amplitude $h^{-1/2}$ and frequency $h^{-1}$ of our initial data are quantized.}

\item The second challenge is that the NLS is only first order in the time variable. Consequently, when analyzing the asymptotic expansion of our ansatz in the weakly-nonlinear regime (inspired by the approach of \cite{Wave_Paper}), the resulting Cauchy transport problem is degenerate as it lacks a time derivative term. This is unlike the case of the wave equation \cite{Wave_Paper}, as there the equation is second order in time and thus their transport equation is non-degenerate as an initial-value problem. To surmount this difficulty, we \textit{borrow} a time derivative with a different power of $h$, resulting in an asymmetry in the transport equation. To resolve this problem, we modify our ansatz so that the space and time-dependent coefficients in our expansion are rapidly varying in time, that is they are functions of $t/h$, with first derivatives in time formally on the order of $1/h$ (for details, see \eqref{3-aA}). More precisely, we assume that our solution of the equation \eqref{generic NLS} takes the form of a truncated asymptotic expansion
\begin{align}\label{approximate solution introduction}
    u(t,x)=h^{-1/2} e^{i\frac{x\cdot \xi}{h} - i\frac{t|\xi|^2}{h^2}} \left(A_0(t/h, x) + h \, A_1(t/h, x) +\ldots + h^N A_N(t/h, x)\right).
\end{align}

% Note that after rescaling $t\mapsto th$, the above relation is completely equivalent to
% \begin{align}
%     u(th, x)= h^{-1/2} e^{i\frac{x\cdot \xi}{h} - i\frac{t|\xi|^2}{h}} \left(A_0(t, x) + h \, A_1(t, x) +\ldots\right),
% \end{align}
% where now the coefficients in the asymptotic expansion of $u$ are independent of $h$.

To adapt to our rapidly varying in time solution, we only evolve our solution over a small time scale on the order of $h$. From the ansatz for $u(t,x)$, we also generate a hierarchy of equations for each $A_j(t,x), j=0, 1,  \ldots, N$ that is different than the one present in \cite{Wave_Paper}. In particular, we treat the leading order equation as a nonlinear transport equation and the $\Delta$ term as a forcing term in the next highest order transport equation, and we repeat this process. These adaptations let us control the higher order terms, and therefore show that the approximate solution closely matches the true solution to the point where the recovery of $\alpha$ is possible.
% Furthermore, the short-time nature of our analysis is advantageous for a numerical verification of our result as the solution must only be
\par 
\end{itemize}

\subsection{Organization of the paper}
We divide the paper in several parts.  In Section \ref{Preliminary Results}, we present some  preliminary results on the global well-posedness and stability of our $\alpha$-cubic NLS in the Sobolev spaces $H^{d/2+\delta}(\R^d)$ for $\delta>0$. In Section \ref{Construction of Approximate Solution}, we solve first the equations for $A_k(t,x)$ in \eqref{approximate solution introduction} for arbitrary $k\ge 0$ in the relevant Sobolev spaces, and derive a stability estimate showing our approximate solution is close to the exact one. In Section \ref{Recovery}, we use the stability results to recover $\alpha$. In Appendix \ref{Appendix}, we prove Proposition 2.3, 2.4 and 2.6.

\subsection{Notation}
We will sometimes write that $A\lesssim B$ to mean that there exist $C>0$ such that $A\le CB$. If $C$ depends on a parameter $\alpha$, we write $A\lesssim_\alpha B$. .

\subsection*{Acknowledgments}
Z.L. gratefully acknowledges support from the NSF under grants No.
DMS-1840314 and DMS-2052789 as well as The University of Texas at Austin for the  Provost's Graduate Excellence Fellowship.
N.P. gratefully acknowledges support from the NSF under grants No.
DMS-1840314 and DMS-2052789.

\section{Preliminary Results}

\label{Preliminary Results}
While existence, uniqueness and stability of solutions to cubic NLS in $\R^d$ for $d\in\{1,2,3\}$ is a well-studied problem, since our initial value problem \eqref{alpha NLS} has a function $\alpha$ included in front of the nonlinearity, in order to make the paper self-contained, in this section we state results pertaining to existence, uniqueness and stability of solutions to \eqref{alpha NLS}.

We start by stating the definition of a strong solution to the initial-value problem \eqref{alpha NLS}.

\begin{defn}\label{Strong solution}
    Let $h$ be as in the statement of Theorem \ref{Main Theorem} let $I\subset \R$. We say that $u(t,x)$ is a \textit{strong solution} in $C^0 H^{d/2+\delta}(I\times \R)$ of 
    \begin{align}
   \begin{cases} \label{alpha NLS 2}
   i\partial_t u + \Delta u=\alpha(x) u|u|^2 \\
    u(-Th,x) = \phi(x)
   \end{cases}
\end{align}
 on an interval $I$ containing $-Th$, if  $u \in  C_t^0 H^{d/2+\delta}(I\times \R)$ and it satisfies
      \begin{align*}
        u(t)= e^{i(t+Th)\Delta} \phi(x) -i\int_{-Th}^t e^{i(t-s)\Delta} \alpha(x) u|u|^2(s)\,ds
    \end{align*}
    for all $t\in I$. If $I$ can be taken to be $[-Th, \infty)$, then we say the strong solution is \textit{global}.
\end{defn}

In order to obtain the results stated in the introduction, we need to work with a solution $u$ to the defocusing cubic $\alpha$-NLS 
equation \eqref{alpha NLS} which exists for sufficiently long time.
% to exist for a sufficiently long time (longer than that is guaranteed by the local theory), we are going to 
Our approach is to first establish a unique global solution to \eqref{alpha NLS}  (in the sense of Definition \ref{Strong solution}) 
in the energy space $C_t^0 H^1_x([-Th, \infty)\times \R^d)$. Then 
by using a persistence of regularity argument, we obtain a solution in $C^0_t H^{d/2+\delta}_x([-Th,\infty)\times \R^d)$.

We first state a result guaranteeing existence of a unique global strong solution in $C^0_t H^1_x$ to \eqref{alpha NLS}.
\begin{prop}\label{GWP H1}
    Let $\phi\in H^1(\R^d), d\in\{1,2,3\}, \alpha\in C_c^\infty(\R^d)$ with $\alpha\ge 0$. Then, there is a unique global strong solution $u$ in $C_t^0 H^1_x([-Th, \infty)\times \R^d)$
    of \eqref{alpha NLS 2}. Furthermore, 
$$\norm{u}_{C^0_t H^1_x([-Th, \infty) \times \R^d)} \lesssim \norm{\phi}_{H^1_x(\R^d)}.$$
\end{prop}

The next result follows by a persistence of regularity argument. 
\begin{prop}\label{GWP H^{d/2+}}
    Let $\phi\in H^{d/2+\delta}_x(\R^d)$, $\alpha\in C_c^\infty(\R^d)$ with $\alpha \ge 0$ and $\delta(d)$ as in the statement of Theorem \ref{Main Theorem}. Then, there is a unique global strong solution $u$ in $C_t^0 H^{d/2+\delta}_x([-Th, \infty)\times \R^d)$
    of \eqref{alpha NLS 2}.
\end{prop}
We also have a stability result, of the type that can be found in \cite{Cazenave}. 
\begin{prop}[Stability]\label{Stability}
    If $u$ is a strong $C_t^0 H^{d/2+\delta}_x([-Th,Th]\times \R^d) $ solution to \eqref{alpha NLS} and $u_N$ is a strong $C_t^0 H^{d/2+\delta}_x([-Th,Th]\times \R^d) $ solution to 
    \begin{align*}
    \begin{cases}
        i\partial_t u_N + \Delta u_N - \alpha u_N|u_N|^2 = R_N & \\
        u_N(-Th,x) = h^{-1/2} e^{ix\cdot \xi/h +iT|\xi|^2/h} \psi(x)
        \end{cases}
    \end{align*}
then we have that
\begin{align*}
    \norm{u-u_N}_{C_t^0 H^{d/2+\delta}_x([-Th,Th]\times \R^d)} \lesssim_{d, \delta, T,N} h\norm{R_N}_{C_t^0 H^{d/2+\delta}_x([-Th,Th]\times \R^d)}.
\end{align*}
\end{prop}
A sketch of the proofs of Propositions \ref{GWP H1} and \ref{GWP H^{d/2+}} as well as a detailed proof of Proposition \ref{Stability} can be found in Appendix \ref{Appendix}.

\section{Construction of Approximate Solution}\label{Construction of Approximate Solution}
% \subsection{Asymptotic solution}

\par 
In this section we prove Theorem \ref{Main Theorem}. In order to describe our strategy, which is inspired by \cite{S_Barreto_2022}, let 
us assume that the hypotheses of Theorem \ref{Main Theorem} are satisfied. We shall construct an approximate solution of  \eqref{alpha NLS} with an error no larger than a constant multiple of $h$ measured in the $C^0_{t} L^\infty_x ([-hT, hT]\times \R^d)$  norm. To do so, we consider an ansatz in the form of a truncated asymptotic series in powers of $h$, with a leading order amplitude of order $h^{p}$ for $p=-1/2$ to place the solution in a regime where recovery of $\alpha$ is possible. We solve the leading order nonlinear transport equation, followed by the higher order linear equations. 

% We consider the defocusing cubic $\alpha$-NLS in $\R_x^d\times \R_t$.
% \begin{align}\label{NLW}
%    P(u):= i\partial_t u + \Delta u-\alpha(x) u|u|^2=0.
% \end{align}
% with $\alpha\ge 0, \alpha\in C_c^\infty(\R^d)$, $\operatorname{supp}(\alpha)\subseteq B(0,T)$. and an incoming wave 
% \begin{align}
%     u(-Th,x) = h^{-1/2} e^{i\frac{x\cdot \xi}{h} - i\frac{t|\xi|^2}{h^2}} \psi(x)\bigg|_{t=-Th}
% \end{align}
% with $\psi\in H^{d/2+\delta+2}_x(\R^d)$ and $\xi\in \R^d$ and
% \begin{align}
%     0<h\ll 1, h\in \frac{3T|\xi|^2}{ 2\pi \N}.
% \end{align}
We first consider an ansatz of the form
\begin{align*}
    u(t,x)= e^{i\phi_h(t,x)} h^p a(t,x, \xi, h)
\end{align*}
with $\phi_h(t,x) = \phi(h^{-2}t, h^{-1} x)$ and some to be determined function $a$.  Substituting this ansatz in \eqref{alpha NLS} and multiplying the equation by $e^{-i\phi_h }$, we find that $a$ must satisfy
\begin{align}\label{3-a}
 e^{-i\phi_h }h^pa(i\partial_t + \Delta) e^{i\phi_h} + h^p(i\partial_t + \Delta) a + 2 ih^{p-1} (\nabla \phi)_h \cdot \nabla a - h^{3p} \alpha(x) a |a|^2=0 
    % &=i h^{p-2} a\left(i(\partial_t \phi)_h + (\Delta \phi)_h + i|(\nabla \phi)_h|^2 \right) +h^p (i\partial_t + \Delta)a + 2 i h^{p-1} (\nabla \phi)_h \cdot \nabla a - h^{3p} \alpha(x) a|a|^2.
\end{align}
In this paper, we are interested only in the weakly nonlinear case, that is when the transport term $\nabla \phi\cdot \nabla a$ is of the same order as the nonlinearity, that is $3p=p-1 \implies p=-1/2$. This will be the regime where we will recover $\alpha(x)$ (in Section \ref{Recovery}).  \par 
From now on, let $p=-1/2$ and 
\begin{align*}
    \phi_h(t,x) = \frac{x\cdot \xi}{h} - \frac{t|\xi|^2}{h^2}.
\end{align*}
Then, we notice that
\begin{align*}
    (i\partial_t + \Delta) e^{i\phi_h}=0.
\end{align*}
Hence, our equation \eqref{3-a} for the amplitude $a$ becomes 
\begin{align}\label{equation for a}
    h^{-1/2} (i\partial_t + \Delta)a + 2 i h^{-3/2} \xi \cdot \nabla a - h^{-3/2} \alpha(x) a|a|^2=0.
\end{align}
We now make the ansatz that $a$ obeys the following truncated asymptotic expansion with respect to $h$,
\begin{align} \label{a_j equation}
    a(h, t,x)&=  a_0(t,x, h) + h\,a_1(t,x, h) + h^2 a_2(t,x, h) + \ldots + h^N a_N(t,x,h)
\end{align}
for some $N\in \mathbb{N}$. To be consistent with the initial data given in Theorem \ref{Main Theorem}, we should impose the following constraints for $a_k$ at time $t=-Th$:
\begin{align}
    a_0( x, -hT, h) &= \psi(x)\\
    a_k( x, -hT, h) &\equiv 0, \mbox{ for } k\ge 1.
\end{align}

% A_0(t,x/h) + h\,A_1(t,x/h) + h^2 A_2(t,x/h) + \ldots + h^N A_N(t,x/h)
% where  Note that the coefficients $A_j(x, t/h)$ are \textit{rapidly varying} in time, with first derivative formally on the order of $h^{-1}$, but independent of $h$ in the spatial variables. 

We substitute this form into \eqref{equation for a}, and record it in the following way, which will help us decompose the nonlinear equation \eqref{equation for a} into a system consisting of 
a simpler nonlinear equation (for $a_0$) and a family of linear equations:

% \begin{align}
% \begin{split}\label{Family}
%     &h^{-3/2} \left(ih\partial_t a_0 + 2i\xi \cdot \nabla a_0 -\alpha(x) a_0 |a_0|^2 \right) +  \\
%     & h^{-1/2}\left(ih\partial_t a_1 + i 2\xi\cdot \nabla a_1 + \Delta a_0 - \alpha(x) (2a_0 a_0 \overline{a_1} + 2a_0\overline{a_0}a_1) \right) +\\
%     & h^{1/2} \left( ih\partial_t a_2 + i2\xi\cdot \nabla a_2  +\Delta a_1 - \alpha(x) O(a_0 a_2 a_2 + a_1 a_1  a_0) \right) + \\
%     & h^{3/2}\left(ih\partial_t a_3  + i 2\xi\cdot \nabla a_3 + \Delta a_2- \alpha(x) O(a_3 a_0 a_0 + a_2 a_1 a_0 + a_1 a_1 a_1)\right) + \ldots+  \\
%     % \left(f_1(a_0, \overline{a_0}, a_1, \overline{a_1}, a_2, \overline{a_2})a_3  +f_2(a_0, \overline{a_0}, a_1, \overline{a_1}, a_2, \overline{a_2})\overline{a_3} \right) 
%     % &h^{5/2} \left( ih\partial_t a_4 + i2\xi\cdot \nabla a_4 - \alpha(x) \left(v_1(a_0, \overline{a_0}, a_1, \overline{a_1}, a_2, \overline{a_2}, a_3, \overline{a_3})a_4  +v_2(a_0, \overline{a_0}, a_1, \overline{a_1}, a_2, \overline{a_2}, a_3, \overline{a_3})\overline{a_4}\right)\right) +\\
%     &h^{N-3/2}\left(ih\partial_t a_N  + i 2\xi\cdot \nabla a_N + \Delta a_{N-1}-\alpha(x) g(a_N, \overline{a_N}, \ldots, a_0, \overline{a_0}) \right) +\\
%     &h^{N-1/2}\Delta a_N =0.
% \end{split}
% \end{align}
\begin{align}
\begin{split}\label{Family}
    &h^{-3/2} \left(ih\partial_t a_0 + 2i\xi \cdot \nabla a_0 -\alpha(x) a_0 |a_0|^2 \right) +  \\
    & h^{-1/2}\left(ih\partial_t a_1 + i 2\xi\cdot \nabla a_1 + \Delta a_0 - \alpha(x) (2a_0 a_0 \overline{a_1} + 2a_0\overline{a_0}a_1) \right) +\\
    & h^{1/2} \left( ih\partial_t a_2 + i2\xi\cdot \nabla a_2  +\Delta a_1 - \alpha(x) g_2(a_0, a_1, a_2, \overline{a_0}, \overline{a_1}, \overline{a_2})  \right) + \\
    & h^{3/2}\left(ih\partial_t a_3  + i 2\xi\cdot \nabla a_3 + \Delta a_2- \alpha(x) g_3(a_0, a_1, a_2,a_3, \overline{a_0}, \overline{a_1}, \overline{a_2}, \overline{a_3})\right) + \ldots+  \\
    % \left(f_1(a_0, \overline{a_0}, a_1, \overline{a_1}, a_2, \overline{a_2})a_3  +f_2(a_0, \overline{a_0}, a_1, \overline{a_1}, a_2, \overline{a_2})\overline{a_3} \right) 
    % &h^{5/2} \left( ih\partial_t a_4 + i2\xi\cdot \nabla a_4 - \alpha(x) \left(v_1(a_0, \overline{a_0}, a_1, \overline{a_1}, a_2, \overline{a_2}, a_3, \overline{a_3})a_4  +v_2(a_0, \overline{a_0}, a_1, \overline{a_1}, a_2, \overline{a_2}, a_3, \overline{a_3})\overline{a_4}\right)\right) +\\
    &h^{N-3/2}\left(ih\partial_t a_N  + i 2\xi\cdot \nabla a_N + \Delta a_{N-1}-\alpha(x) g_N(a_0, a_1, a_2, \dots, a_N, \overline{a_0}, \overline{a_1}, \overline{a_2}, \dots, \overline{a_N}) \right) +\\
    &h^{N-1/2}\Delta a_N + \sum_{N+1\le k\le 3N} h^{k-3/2} g_k(a_0, a_1, a_2, \dots, a_k, \overline{a_0}, \overline{a_1}, \overline{a_2}, \dots, \overline{a_k})   =0 
\end{split}
\end{align}
% where $g_j, f_j, v_j, j=1,2$ are polynomials.
where 
\begin{align}
    g_k(a_0, a_1, a_2, \dots, a_k, \overline{a_0}, \overline{a_1}, \overline{a_2}, \dots, \overline{a_k}) = \sum_{\substack{j+m+l=k \\ 0\le j,m,l}}  a_j a_m \overline{a_l}
\end{align}
is the coefficient of $h^k$ in the expansion of the cubic nonlinear term
\begin{align}
    (a_0+ ha_1 + \cdots + h^N a_N) \left|a_0+ ha_1 + \cdots + h^N a_N\right|^2=a_0|a_0|^2 + g_1h + \cdots + g_{3N}h^{3N}.
\end{align}
Note that for each $k$, $g_k$ has in every of its term at most one copy of $a_k$ or $\overline{a_k}$.
% where $g_k$ is a cubic polynomial in $a_j, \overline{a_j}$ for $0\le j\le k$ with at most one copy of $a_k$ or $\overline{a_k}$ appearing in each term. 

In Section \ref{a0}, we solve the equation for $a_0$ (which corresponds to setting the first line of \eqref{Family} to be zero), while in Section \ref{higher order terms Section}, we solve the higher order equations for $a_k, k\ge 1$ (the equation for $a_k$ corresponds to the $k+1$-th line of \eqref{Family} being zero). Then in Section \ref{Recovery}, we perform the recovery of $\alpha(x)$. 
However before we do that we further adjust our ansatz for $a_j$, $j \in \{0,1,\dots,N\}$ as follows: 
\begin{equation}\label{3-aA}
a_j(t,x,h)= A_j(x, t/h)
\end{equation}
Note that the coefficients $A_j(x, t/h)$ are \textit{rapidly varying} in time, with first derivative in time formally of the order of $h^{-1}$, but independent of $h$ in the spatial variables.

\subsection{The leading order equation for  $a_0$ }\label{a0}

% and
% \begin{align}
%     \Delta \phi_h =0.
% \end{align}
% Hence, we are left to solve
% \begin{align}
%     h^{-1/2} (i\partial_t + \Delta)a + 2 i h^{-3/2} \xi \cdot \nabla a - h^{-3/2} \alpha(x) a|a|^2=0
% \end{align}
In this subsection, we study the following initial value problem for $a_0$ which records the above mentioned nonlinear transport equation for $a_0$ and takes the form: 
\begin{align}\label{Equation a0}
    \begin{cases}
        h\partial_ta_0 + 2 \xi \cdot \nabla a_0 =-i\alpha(x) a_0 |a_0|^2 \\
        a_0(x, -hT, h) = \psi(x).
    \end{cases} 
\end{align}
We will achieve that by converting \eqref{Equation a0} into an initial value problem which can be solved by the method of characteristics
\footnote{Actually, the initial value problem \eqref{Equation a0} itself can be solved by the method of characteristics. However, to make the presentation cleaner, we first rescale it in the time variable.}.

Combining \eqref{3-aA} and \eqref{Equation a0} leads to 
\begin{align}\label{tag 1}
        \begin{cases}
        \partial_t A_0 + 2 \xi \cdot \nabla A_0 =-i\alpha(x) A_0 |A_0|^2 \\
        A_0(x, -T) = \psi(x).
    \end{cases} 
\end{align}
Now, we employ the method of characteristics to explicitly solve the above equation. Letting 
\begin{align}\label{b Definition}
b(t)= A_0(x+ 2\xi t,t)
\end{align}
we find that
\begin{align}\label{tag 2}
    \frac{d}{dt} b(t) = \left(\partial_t A_0 + 2\xi \cdot \nabla A_0\right)( x+2\xi t,t) = -i\alpha (x+2\xi t) b |b|^2.
\end{align}
This is a nonlinear, first order, non-autonomous ODE for the function $b(t)$. We first notice that 
\begin{align*}
    \frac{d}{dt} |b|^2 &= 2 \operatorname{Re}(\overline{b(t)} \frac{d}{dt} b(t)) \\
    &=2\operatorname{Re} \left(-i\alpha(x+2\xi t) |b|^4\right) \\
    &=0.
\end{align*}
Thus, we can replace $|b(t)|^2$ in \eqref{tag 2} with $|b(-T)|^2$,  and we obtain the following \textit{linear} ODE. 
\begin{align*}
    \frac{d}{dt} b = -i\alpha (x+2\xi t) |b(-T)|^2 b(t),
\end{align*}
whose solution is given by
\begin{align}\label{b equation}
    b(t) = b(-T) \exp \left(-i |b(-T)|^2 \int_{-T}^t \alpha(x+2\xi s)\,ds\right).
\end{align}
Note that by \eqref{b Definition} and \eqref{tag 1},
\begin{align*}
    b(-T) = A_0( x-2\xi T, -T) = \psi(x - 2\xi T).
\end{align*}
Hence, we find that by \eqref{b equation},
\begin{align}\label{b final}
   b(t)=  \psi(x- 2\xi T) \exp \left(-i |\psi(x - 2\xi  T)|^2 \int_{-T}^t \alpha(x+2\xi s)\,ds\right).
\end{align}
Now, we can reverse the process to solve for $A_0$, and then $a_0$. More precisely, by \eqref{b Definition} and \eqref{b final},
\begin{align*}
    A_0(t,x)= \psi(x  -2 \xi t -2\xi T) \exp \left(-i |\psi(x - 2\xi t - 2\xi  T)|^2 \int_{-T}^t \alpha(x-2\xi t +2\xi s)\,ds\right),
\end{align*}

which together with \eqref{3-aA} implies
\begin{align}\label{a_0 full equation}
    a_0(t,x,h) = \psi(x  -2 \xi t/h -2\xi T) \exp \left(-i |\psi(x - 2\xi t/h - 2\xi T)|^2 \int_{-T}^{t/h} \alpha(x-2\xi t/h +2\xi s)\,ds\right).
\end{align}
In order to control the higher order terms, we need to estimate the norm $\norm{a_0}_{C^0_t H_x^{d/2+2N+\delta+2} ([-Th, \infty]\times \R^d)}$ for $\delta>0$ and $N\ge 1$. We state the estimate in the following lemma.
\begin{lem}\label{Lemma for a_0}
    Let $d\in\{1,2,3\}$. Let $a_0(t,x,h)$ be as defined in \eqref{a_0 full equation}, $T, R>0$, $\operatorname{supp}(\psi)\subseteq B(0,R), \operatorname{supp}(\alpha)\subseteq B(0,T_0)\subseteq B(0,T)$, and $|\xi|=1$. Additionally, let $\delta(d)$ be defined as in the statement of Theorem \ref{Main Theorem}. We have the following global Sobolev bound
    \begin{align}\label{Lemma for a_0}
            \norm{ a_0(t,x))}_{C_t^0 H^{d/2+2N+\delta }_x([-Th, \infty)\times \R^d)} \le C(d,T,R, \norm{\alpha}_{W_x^{2N+2, \infty}(\R^d)}, \norm{\psi}_{W_x^{2N+2, \infty}(\R^d)}),
    \end{align}
for any $N\ge 1$.
\end{lem}
\begin{proof}
We will provide details of such an estimate in the case when $N=1$. The other cases can be done in a similar fashion. \par

 Note that for $d\in\{2,3\}$, we have that $d/2+2+\delta\le 3/2+2+\delta\le 4$ for $0\le \delta\le 1/2$. If $d=1$, then with $1/2\le \delta <1$, we have that $d/2+2+\delta \le 1/2+2+1\le 4$ Thus, in any case,
\begin{align}\label{H^4 estimate for a_0}
    \norm{a_0}_{L^\infty_t H_x^{d/2+2+\delta+2} [-Th, \infty]\times \R^d)}\lesssim \norm{a_0}_{L^\infty_t L^2_x [-Th, \infty]\times \R^d)} + \sum_{\alpha, |\alpha|=4} \norm{\partial^\alpha  a_0}_{L^\infty_t L^2_x [-Th, \infty]\times \R^d)},
\end{align}
where the sum is over all multi-indices $\alpha$ such that $|\alpha|=4$.
Let us write 
\begin{align}
    a_0=a_0(t,x,h)= \varphi \exp\left(-i\chi\right),
\end{align}
where 
% $\varphi , \chi$ can be calculated from the expression \eqref{a_0 full equation} for $a_0$.
\begin{align}
    \varphi(t,x,h) &:= \psi(x  -2 \xi t/h -2\xi T) \\
    \chi(t,x,h) &:= |\psi(x - 2\xi t/h - 2\xi T)|^2 \int_{-T}^{t/h} \alpha(x-2\xi t/h +2\xi s)\,ds.
\end{align}
For simplicity of exposition, let us focus on multi-indices $\alpha$ such that $\partial^\alpha=\partial_{x_l}^4$ for some $1\le l\le d$ . Note that in this case, 
\begin{align}\label{Product rule on a_0}
\partial^4_{x_l} a_0 = \sum_{j=0}^4 {4 \choose j} \partial_{x_l}^{j} \varphi \,\partial_{x_l}^{4-j} \exp\left(-i\chi\right) =: \exp(-i\chi) \left(\partial^4_{x_l}  \varphi + \sum_{j=0}^{3} {4\choose j}\partial^{j} \varphi \,P_j(\partial_{x_l}\chi, \dots, \partial_{x_l}^4 \chi)\right)
    % \partial_x a_0 &= (\varphi_x-i\varphi \chi_x)\exp(-i\chi) \\
    % \partial_{xx} a_0 &= (\varphi_{xx}-i\varphi \chi_{xx} -i \varphi_x \chi_x)\exp(-i\chi) -i \chi_x (\varphi_x-i\varphi \chi_x)\exp(-i\chi),
\end{align}
where $P_j$ is some polynomial such that for each $0\le j\le 3$, $P_j$ is non-constant. Our strategy will be as follows: since $\partial_{x_l}^4 a_0$ is compactly supported in a ball of radius $R$ for each $t$, we estimate, by the triangle inequality,
\begin{align}
    \norm{\partial_{x_l}^4 a_0}_{L^2_x}\le C(d) R^{d/2} \norm{\partial_{x_l}^4 a_0}_{L^\infty_x}
\end{align}

Applying the triangle inequality to \eqref{Product rule on a_0},
\begin{align}\label{triangle inequality for a_0 estimate}
    \norm{\partial_{x_l}^4 a_0}_{L^\infty_x} \le \norm{\partial_x^4 \psi}_{L^\infty_x} + \sum_{j=0}^3 {4\choose j} \norm{\partial_{x_l}^j \varphi \,P_j(\partial_{x_l}\chi, \dots, \partial_{x_l}^4 \chi)}_{L^\infty_x}.
\end{align}
To estimate the terms in the summand, we need to understand to which region the integral is restricted to as both $\psi$ and $\alpha$ are compactly supported and appear in the integrand. Let us present the case $j=3$ as the others are quite similar. Note that, using our hypotheses on the supports of $\psi$ and $\alpha$, 
\begin{align}
    |\partial_{x_l}^3\varphi  P_3(\chi)| &= \left|\partial_{x_l}^3\varphi \int_{-T}^{t/h}\left( 2 \re (\overline{\varphi} \partial_{x_l} \varphi)  \alpha(x(s)) + |\varphi|^2  \partial_{x_l} \alpha(x(s))\right) \,ds\right| \\
    &\lesssim \norm{\varphi}_{W^{3,\infty}_x}^3 \int_{-T}^{t/h}  1_{|x-2\xi t/h-2\xi T|\le R}(|\alpha_{x_l}(x(s))| + |\alpha(x(s))|) \\
    &\le \norm{\varphi}_{W^{3,\infty}_x}^3 \norm{\alpha}_{W^{1, \infty}}
     \int_{-\infty}^{\infty}  1_{|x-2\xi t/h-2\xi T|\le R}1_{\{|x-2\xi t /h +2\xi s|\le T\}}\, ds \\
     &= \norm{\psi}_{W^{3,\infty}_x}^3
     \norm{\alpha}_{W^{1, \infty}} |A_{T, R,\xi}| \label{Estimate 1 for a_0}
\end{align}
% where we have used the facts that $\psi$ is supported in $B(0,R)$ and 
% $\psi\in W^{3,\infty}_x(\R^d)$, and $x(s)$ denotes $x-2\xi t/h+2\xi s$.

% The delicate terms to estimate will be the ones in the summand where at least one factor of $\chi_x$ or a derivative thereof appears as the naive bound
% \begin{align}
%     \int_{-T}^{t/h} \alpha(x-
% \end{align}

% We recall that $\psi$ is compactly supported such that its support lies in $B(0,R)$. Thus, any integral of $a_0$ will be restricted to the region $|x-2\xi t/h-2\xi t|\le R$. If we differentiate $a_0$ with respect to $x$, then we will have to estimate, for example, 

% We now estimate the second term on the right hand side of \eqref{Estimate 1 for a_0} and remark that the first term can be treated similarly. We wish to estimate 
% % In order to recursively solve for $a_j$, $j \in \{1, \dots, N\}$, we shall find an estimate on 
% \begin{align*}
%     \int_{-T}^{t/h}1_{|x-2\xi t/h-2\xi T|\le R}(x) |\alpha(x-2\xi t/h+2\xi s)|\,ds
% \end{align*}
% for large values of $t/h$ and independently of $h$. 

 % Now, if $x\in  B(2\xi t/h+2\xi T,R)$, then let us find the range of values of $s$ such that $x(s)=x-2\xi t/h+2\xi s$ stays inside $B(0,T)$, wherein the support of $\alpha$ lies. If we let $y=x-2\xi t/h$, we estimate the size of the set
where 
\begin{align*}
    A_{T, R,\xi} := \{s:  |y-2\xi T|\le R \text{ and } |y+2\xi s|\le T\}.
\end{align*}
for $y=x-2\xi t/h$.
If $s\in A_{T,R, \xi}$, by the triangle inequality, 
\begin{align*}
    |2\xi T+2\xi s| &= |(y-2\xi T)-(2\xi s+y)| \\
    &\le R+T.
\end{align*}
Therefore, since $|\xi|=1$,
\begin{align*}
    A_{T, R, \xi} \subseteq \left\{s: |s+T|\le \frac{R+T}{2}\right\}
\end{align*}
and
\begin{align*}
    |A_{T,R, \xi}|\le R+T.
\end{align*}
% Clearly, we must have
% \begin{align}
%     2|\xi| |t/h-s|\le  2R, 
% \end{align}
% since otherwise we can show that $x-2\xi(t/h-s)\notin B(2\xi t/h-2\xi T,R)$. Thus, 
% \begin{align}
%     -\frac{R}{|\xi|} + t/h\le s\le \frac{R}{|\xi|} + t/h. 
% \end{align}

% Thus, we can bound
% \begin{align*}
%     \int_{-T}^{t/h}1_{|x-2\xi t/h-2\xi T|\le R}(x) |\alpha(x-2\xi t/h+2\xi s)|\,ds &\le \int_{-\infty}^{\infty }1_{|x-2\xi t/h-2\xi T|\le R}(x) |\alpha(x-2\xi t/h+2\xi s)|\,ds\\
%     &\le \norm{\alpha}_{L^\infty_x(\R^d)} |A_{T, R, \xi}| \\
%     &\le \frac{R+T}{|\xi|} \norm{\alpha}_{L^\infty_x(\R^d)}
% \end{align*}
% independently of $t,h$.

Hence, \eqref{Estimate 1 for a_0} implies 
\begin{align}
    \norm{\partial_x^3 \varphi P_3}_{L^\infty_x } \lesssim \norm{\psi}_{W^{3, \infty}_x}^3 (R+T) \norm{\alpha}_{W^{1,\infty}_x(\R^d)}.
\end{align}Similarly, we can show that for each $j$, 
\begin{align}
    \norm{\partial_x^{j} \varphi P_j}_{L^\infty_x} \le C( R, T,\norm{\alpha}_{W^{4,\infty}_x(\R^d)},\norm{\psi}_{W^{4, \infty}_x} ).
\end{align}
Therefore, recalling \eqref{triangle inequality for a_0 estimate}, we obtain 
\begin{align}
    \sup_{t\in[-Th, \infty)}\norm{\partial_{x_l}^4 a_0}_{L^2_x}\le C(d,R, T,\norm{\alpha}_{W^{4,\infty}_x(\R^d)},\norm{\psi}_{W^{4, \infty}_x} )
\end{align}
and the same estimate holds for $\partial^\alpha a_0$ for any multi-index $\alpha$ such that $|\alpha|=4$.
% Now, note that $\frac{d}{2}+2+\delta \le3/2+2+\delta\le 4$ if $0<\delta<1/2$ since $d\le 3$.

% Hence the $H^{d/2+2+\delta}_x(\R^d)$ norm of $a_0$

% is controlled by the $H^4_x(\R^d)$ norm of $a_0$. 
On the other hand, by applying the triangle inequality to \eqref{a_0 full equation}, we can easily obtain the following bound on the $L^2$ norm of $a_0$ 
\begin{align*}
    \norm{a_0}_{L^\infty_t L^2_x([-Th, \infty)\times \R^d)} \le C(d)\norm{\psi}_{L^\infty_x} R^{d/2}.
\end{align*}
%  By our above calculations we have that any multi-index $\alpha$ with $|\alpha|=4$,
% \begin{align*}
%     \norm{\partial^{\alpha} a_0(t,x)}_{C_t^0L^2_x([-Th, \infty)\times \R^d)}\le C(d,T,R, \norm{\alpha}_{W_x^{4, \infty}(\R^d)}, \norm{\psi}_{W_x^{4, \infty}(\R^d)}).
% \end{align*}
Hence, going back to \eqref{H^4 estimate for a_0}, we obtain
\begin{align*}
    \norm{ a_0(t,x))}_{L^\infty_t H^{d/2+2+\delta }_x([-Th, \infty)\times \R^d)} &\le \norm{ a_0(t,x))}_{L^\infty_t H^{4}_x([-Th, \infty)\times \R^d)} \\
    &\le C(d,T,R, \norm{\alpha}_{W_x^{4, \infty}(\R^d)}, \norm{\psi}_{W_x^{4, \infty}(\R^d)}).
\end{align*}
Similarly, for $N\ge 1$ a positive integer,
\begin{align} \label{a_0 H^s}
    \norm{ a_0(t,x))}_{L^\infty_t H^{d/2+2N+\delta }_x([-Th, \infty)\times \R^d)} \le C(d,T,R, \norm{\alpha}_{W_x^{2N+2, \infty}(\R^d)}, \norm{\psi}_{W_x^{2N+2, \infty}(\R^d)}).
\end{align}
Note that we can actually prove that $a_0(t,x)\in C_t^0 H^{d/2+2N+\delta }_x$ with similar bounds above.
\end{proof}
\begin{rmk}\label{Remark on C}
    From now on, we will call the constant appearing in \eqref{Lemma for a_0} $C(d, T, R, \alpha, \psi, N)$. 
\end{rmk}

\subsection{Higher order terms} \label{higher order terms Section}

For the equation at order $h^{-3/2}$ in \eqref{Family}, we already explicitly solved the nonlinear transport equation
\begin{align*}
    \begin{cases}
        ih\partial_t a_0 + 2i\xi \cdot \nabla a_0 -\alpha(x) a_0 |a_0|^2 =0 \\
        a_0(-Th, x) = \psi(x).
    \end{cases}
\end{align*}
At order $h^{-1/2}$, by \eqref{Family}, we must solve
\begin{align*}
        \begin{cases}
        ih\partial_t a_1  + i 2\xi\cdot \nabla a_1 +\Delta a_0 - \alpha(x) (2a_0 a_0 \overline{a_1} + 2a_0\overline{a_0}a_1) = \\
        a_1(-Th, x) = 0.
    \end{cases}
\end{align*}
At order $h^{1/2}$, by \eqref{Family}, we must solve
\begin{align} \label{equation for a_2}
        \begin{cases}
        ih\partial_t a_2 + i2\xi\cdot \nabla a_2  +\Delta a_1 - \alpha(x) g_2(a_0, a_1, a_2, \overline{a_0},\overline{a_1},\overline{a_2} )=0 \\
        a_2(-Th, x) = 0,
    \end{cases}
\end{align}
etc.
% Note that $a_2(t,x)\equiv 0$ is a solution, and is the only one if uniqueness is true.  \newline
% At order $h^{3/2}$, by \eqref{Family}, we must solve
% \begin{align}
%         \begin{cases}
%         (i\partial_t + \Delta) a_3  + i 2\xi\cdot \nabla a_3 - \alpha(x)\left(f_1(a_0, \overline{a_0}, a_1, \overline{a_1}, a_2, \overline{a_2})a_3  +f_2(a_0, \overline{a_0}, a_1, \overline{a_1}, a_2, \overline{a_2})\overline{a_3} \right)=-\Delta a_2 \\
%         a_3(-Th, x) = 0.
%     \end{cases}
% \end{align}
% Note that if $a_2(t,x)\equiv 0$, then $a_3\equiv 0$ is a solution.\newline
% At order $h^{5/2}$, by \eqref{Family}, we must solve
% \begin{align}
%         \begin{cases}
%         ih\partial_t a_4 + i2\xi\cdot \nabla a_4 - \alpha(x) \left(h_1(a_0, \overline{a_0}, a_1, \overline{a_1}, a_2, \overline{a_2}, a_3, \overline{a_3})a_4  +f_2(a_0, \overline{a_0}, a_1, \overline{a_1}, a_2, \overline{a_2}, a_3, \overline{a_3})\overline{a_4}\right)=0 \\
%         a_4(-Th, x) = 0.
%     \end{cases}
% \end{align}
% Again, $a_4(t,x)\equiv 0$ is a solution. Continuing in this way, if we have uniqueness for each equation, we can show that
% \begin{align}
%     a_k(t,x) \equiv 0, k\ge 2.
% \end{align}
% This would mean that if we solve a finite number $N$ of these equations for $a_j$, the function
% \begin{align}
%     a(t,x)= a_0(t,x) + ha_1(t,x)
% \end{align}
%  solves the equation for the amplitude $a(t,x)$ up to an error with $L^\infty_x$ amplitude proportional to $h^{N-3/2}$. 

Let us first define our notion of strong solution for the equation for $a_k, k\ge 2$.
\begin{defn}
    We say that $a_k(t,x)$ is a strong $C^0_t H^{d/2+\delta}(I\times \R^d)$  solution of
    \begin{align}
        \begin{cases}
            h\partial_t a_k + 2\xi \cdot \nabla a_k=F(t,x, a_k(t,x)) \\
            a_k(x,-Th)=0
        \end{cases}
        \end{align}
        on an interval $I$ containing $-Th$ if the rescaled function defined by $A_k(t/h,x):=a_k(t,x)$ satisfies $A_k(t,x)\in C^0_t H^{d/2+\delta}(Ih\times \R^d)$ and for each $t'\in hI$
        \begin{align}
            A_k(t',x(t')) =- \int_{-T}^{t'} F(s', x(s'), A_k(s',x(s')))\,ds,
        \end{align}
        where $x(t)=x+2\xi t$, $t'=t/h, s'=s/h$.
\end{defn} 
Let us  now tackle the equation at order $h^{-1/2}$ and state our result as a proposition.
\begin{prop}
    Let $d\in\{1,2, 3\}$. Given  $T>0$, there exists a unique $C_t^0 H^{d/2+\delta}_x$ strong, global  solution $a_1(t,x)$ to the equation
    \begin{align}\label{a1 equation higher order terms}
        \begin{cases}
        ih\partial_t  a_1  + i 2\xi\cdot \nabla a_1 - \alpha(x) (2a_0 a_0 \overline{a_1} + 2a_0\overline{a_0}a_1) =-\Delta a_0 \\
        a_1(-Th, x) = 0.
    \end{cases}
    \end{align}
    on the interval $[-Th,\infty)$ (and the interval $(-\infty, -Th)$ by time-reversal symmetry) with the bound
    \begin{align*}
            \norm{a_1(t,x)}_{C_t^0 H^{d/2+\delta}_x([-hT, hT]\times  \R^d)} &\le C(d,T, \alpha, \psi, 1) \exp\left(2T C(d,\delta)^3 \norm{\alpha}_{H^{d/2+\delta}_x(\R^d)} C(d,T, \alpha, \psi, 1)^2\right)
    \end{align*}
    on the interval $[-Th, Th]$.
\end{prop}
\begin{proof}
Let $A_1, A_0$ be as defined in \eqref{3-aA}, that is $a_1(t,x)=A_1(t/h,x)$, $a_0(t,x)=A_0(t/h,x)$. Then, equation \eqref{a1 equation higher order terms} takes the equivalent form
\begin{align}\label{A_1 equation}
           \begin{cases}
        \partial_{t'}  A_1  +  2\xi\cdot \nabla A_1 +i \alpha(x) (2A_0 A_0 \overline{A_1} + 2A_0\overline{A_0}A_1) =i\Delta A_0 \\
        A_1(-T, x) = 0.
    \end{cases}
\end{align}
with $A_1, A_0$ evaluated at $t'=t/h$. Let $$A_1(t', x(t')):=A_1(t', x+2\xi t').$$
Then, by the chain rule, 
\begin{align*}
    \frac{d}{dt'} A_1(t', x(t'))= (\partial_{t'} A_1 + 2\xi \cdot \nabla A_1)\big|_ {(t', x+2\xi t')}.
\end{align*}
Hence, performing the change of variable  $x\mapsto x+2\xi t'$ on the equation \eqref{A_1 equation}, we obtain the equation
\begin{align}\label{translated equation for A1}
        \begin{cases}
        \frac{d}{dt'} A_1(t', x(t'))    = -4i\alpha(x+2\xi t')  A_0(t',x+2\xi t') \re (A_0(t', x+2\xi t') \overline{A_1(t', x(t'))}) +i\Delta A_0(t',x+2\xi t') \\
        A_1(-T,x(-T)) = 0.
    \end{cases}
\end{align}
In Duhamel form, \eqref{translated equation for A1} takes the form
\begin{align} \label{Duhamel A_1}
    A_1(t', x(t'))= -i\int_{-T}^{t'}\left(4\alpha(x+2\xi s)  A_0(s,x+2\xi s) \re (A_0(s, x+2\xi s) \overline{A_1(s, x(s)}) -\Delta A_0(s,x+2\xi s)\right)\,ds
\end{align}
and we are looking for a fixed point of $G$, where $G$ is defined as 
\begin{align*}
    G[ A_1(t',x(t'))]\equiv  -i\int_{-T}^{t'} \left(4\alpha(x+2\xi s)  A_0(s,x+2\xi s) \re (A_0(s, x+2\xi s) \overline{A_1(s,x(s))}) -\Delta A_0(s,x+2\xi s)\right)\,ds.
\end{align*}
We proceed via Banach's-fixed point theorem. Define the Banach space
\begin{align*}
    B := \{u\in C^0_t H_x^{d/2+\delta} ([-T, -T+S) \times \R^d): \norm{u}_{C^0_t H^{d/2+\delta} ([-T, -T+S) \times \R^d)}\le  1\}
\end{align*}
for $S>0$ and let $A_1\in B$. Then, by Lemma \ref{Algebra property} and Lemma \ref{Lemma for a_0} (which we apply to $A_0(t',x)$ by a rescaling in time)
% and assuming that $S\le 4Th$ (this is harmless for the recovery as we will only need the solution to exist for a length of time $2Th$),
we estimate the $C^0_t H^{d/2+\delta} ([-T, -T+S) \times \R^d)$ norm of $G[ A_1]$ as follows:
\begin{align*}
    &\norm{G[ A_1]}_{L^\infty_t H^{d/2+\delta} ([-T, -T+S) \times \R^d)} \\
    &\le 4S \,C(d, \delta)^3 \norm{\alpha}_{H^{d/2+\delta}(\R^d)} \norm{A_0}_{C^0_t H^{d/2+\delta} [-T, -T+S]\times \R^d)}^2 \\
    &+ S \norm{A_0}_{C^0_t H^{d/2+2+\delta} ([-T, -T+S] \times \R^d)} \\
    &\le  4S \,C(d, \delta)^3\norm{\alpha}_{H_x^{d/2+\delta}(\R^d)} C(d, T, R, \alpha, \psi, 1)^2 + S \,C(d, T, R, \alpha, \psi, 1) \\
    &\le 1,
\end{align*}
where the last inequality holds provided that
\begin{align*}
    0\le S \le S^*\equiv C(d, T, R,  \alpha, \psi, 1)^{-1} \min\left(\frac{1}{8}C(d, \delta)^{-3}\norm{\alpha}_{H_x^{d/2+\delta}}^{-1}C(d, T, R, \alpha, \psi, 1)^{-1}, \frac{1}{2}\right)
\end{align*}
with $C(d, T, R, \alpha, \psi, N)$ as in Remark \ref{Remark on C}. \par 
We next show that the quantity $\norm{G[A_1(t', x(t'))]}_{H^{d/2+\delta}_x(\R^d)} $ is continuous in $t'$. First note that by a corollary of the triangle inequality for Banach spaces,  
\begin{align*}
    |\norm{G[A_1(t_1, x(t_1))]}_{H^{d/2+\delta}_x(\R^d)}-\norm{G[A_1(t_2, x(t_2))]}_{H^{d/2+\delta}_x(\R^d)}| \le \norm{G[A_1(t_1, x(t_1))-G[A_1(t_2, x(t_2))}_{H^{d/2+\delta}_x(\R^d)}.
\end{align*}
Now, we can use Duhamel's formula to estimate the above quantity as follows, for $t_1, t_2\in[-T,-T+S^*]$, 
\begin{align*}
    \norm{G[A_1(t_1, x(t_1))-G[A_1(t_2, x(t_2))}_{H^{d/2+\delta}_x(\R^d)} \le |t_1-t_2|\left(C(d,\delta)^3 C(d, T, R, \alpha, \psi, 1)^2 +  C(d, T, R, \alpha, \psi, 1)\right).
\end{align*}
% and $0\le S\le 4Th$. Note that for small enough $h$, this condition is equivalent to just asking that $0\le S\le 4Th$.
Thus, $\norm{G[A(t', x(t'))]}_{H^{d/2+\delta}_x(\R^d)}$ is Lipschitz-continuous in $t'$ and thus continuous in $t'$. 
Hence, for $S$ as above,
\begin{align*}
    G: B\to B.
\end{align*}
Now, we have that for $A_1,A_2 \in B$, by Duhamel and Proposition \ref{Algebra property} again,
\begin{align*}
    \norm{G[A_1]- G[A_2]}_{C^0_t H^{d/2+\delta} ([-T, -T+S] \times \R^d)} &\le S \,C(d, \delta)^3\norm{\alpha}_{H^{d/2+\delta}(\R^d)}\norm{A_0}_{C^0_t H^{d/2+\delta} ([-T, -T+S]\times \R^d)}^2 \norm{A_1-A_2}_{C^0_t H^{d/2+\delta} ([-T, -T+S] \times \R^d)}\\
    &\le \frac{1}{2} \norm{A_1-A_2}_{C^0_t H^{d/2+\delta} ([-T, -T+S] \times \R^d)}.
\end{align*}
where the last inequality holds as long as
\begin{align*}
    0\le S\le  \frac{1}{2} C(d,\delta)^{-3} C(d, T, \alpha, \psi, 1)^{-2} \norm{\alpha}_{H_x^{d/2+\delta}}^{-1},
\end{align*}
which is satisfied provided $0\le S\le S^*$.

Hence, we have shown that $G$ is a contraction for small enough values of $S$. Thus, for $0\le S\le S^*$, by the Banach fixed-point theorem, there exists a unique solution $A_1(t, x(t))$ to the desired equation \eqref{Duhamel A_1} in $B$ on the time interval $[-T, -T+S]$ and thus we have a  solution on the interval $[-T, -T+S^*]$. Thus, we can find a strong solution $A_1(t',x)=A_1(t', x(t')-2\xi t')\in B$ to  equation \eqref{A_1 equation} on this time interval. \par 
Note that equation \eqref{A_1 equation} is linear in $A_1$ up to complex conjugation. Thus, the linear theory iterates to give a global solution since the time of existence of the solution is independent of the $H^{d/2+\delta}_x$ norm of the initial data. Note that by taking the $H^{d/2+\delta}_x$ norm of both sides of the  Duhamel formulation \eqref{Duhamel A_1}, we have by the triangle inequality, and the translation invariance of Sobolev norms,
\begin{align*}
\norm{A_1(t',x)}_{H^{d/2+\delta}_x(\R^d)}
% &\le \norm{A_0(t',x)}_{H^{d/2+\delta+2}_x(\R^d)} + C(d,\delta)^3 \int_{-T}^{t'} \norm{\alpha}_{H^{d/2+\delta}_x(\R^d)} C(d,T,R, \alpha, \psi, 1)^2 \norm{A_1(s,x)}_{H^{d/2+\delta}_x(\R^d)}\,ds \\
&\le C(d,T, R,\alpha, \psi, 1)  + C(d,\delta)^3 \int_{-T}^{t'} \norm{\alpha}_{H^{d/2+\delta}_x(\R^d)} C(d,T,R, \alpha, \psi, 1)^2 \\
&\times \norm{A_1(s,x)}_{H^{d/2+\delta}_x(\R^d)}\,ds,
\end{align*}
which by Grönwall's inequality leads to the unconditional bound
\begin{align*}
    \norm{A_1(t',x)}_{H^{d/2+\delta}_x(\R^d)} &\le C(d,T,R, \alpha, \psi, 1) \exp\left(\int_{-T}^{t'} C(d,\delta)^3 \norm{\alpha}_{H^{d/2+\delta}_x(\R^d)} C(d,T, R, \alpha, \psi, 1)^2\right) \\
    &=C(d,T, R, \alpha, \psi, 1) \exp\left((t'+T) C(d,\delta)^3 \norm{\alpha}_{H^{d/2+\delta}_x(\R^d)} C(d,T, R, \alpha, \psi)^2\right)
\end{align*}
for $t'\in[-T, \infty ]$. Thus, undoing the rescaling $t'=t/h$, we have that 
\begin{align*}
    \norm{a_1(t,x)}_{H^{d/2+\delta}_x(\R^d)} \le C(d,T,R,  \alpha, \psi, 1) \exp\left((t/h+T) C(d,\delta)^3 \norm{\alpha}_{H^{d/2+\delta}_x(\R^d)} C(d,T, R, \alpha, \psi, 1)^2\right)
\end{align*}
for $t\in [-Th, \infty)$. Hence, taking the supremum over $[-Th, Th]$, we obtain the bound
% , but by a standard argument, we can continue the fixed-point solution as long as the $H^{d/2+\delta}_x(\R^d)$ norm of $A_1$ is finite, hence the solution $A_1(t,x)$ is global, hence so is $a_1(t,x)=A_1(t/h,x)$

\begin{align*}
    \norm{a_1(t,x)}_{C_t^0 H^{d/2+\delta}_x([-hT, hT]\times  \R^d)} &\le C(d,T,R, \alpha, \psi, 1)\exp\left(2T \, C(d,\delta)^3 \norm{\alpha}_{H^{d/2+\delta}_x(\R^d)} C(d,T, R,\alpha, \psi, 1)^2\right) 
\end{align*}
\par 
\end{proof}
\begin{rmk}\label{remark 1}
    By differentiating equation \eqref{Duhamel A_1} with respect to $x$ twice \textit{before} constructing the contraction, we can in fact obtain similar bounds for the higher order Sobolev norm $\norm{a_1(t,x)}_{C_t^0 H^{d/2+\delta+2}_x([-hT, hT]\times  \R^d)}$, which will then be used when we consider the equation for $a_2$ in \eqref{equation for a_2}.
\end{rmk}
\begin{rmk}
    We can take a similar approach to the equation for $a_2$, constructing a contraction in $C_t^0 H^{d/2+\delta}_x(I \times \R^d)$ for some interval $I$ before extending the solution globally via iteration. Note the importance of Remark \ref{remark 1} in this strategy. We can recursively continue to solve for $a_k(t,x)$ for each $0 \le k\le N$ and obtain the bounds
    \begin{align}
    \norm{a_k(t,x)}_{C_t^0 H^{d/2+\delta}_x([-hT, hT]\times  \R^d)} \lesssim_{d, T, R, \alpha, \psi, k} 1.
    \end{align}
\end{rmk}

\subsection{Proof of Theorem \ref{Main Theorem}}

% {\color{red}
% Recall construction of
% Note that $v$ can be written as follows: $v=h^{-1/2} e^{\cdot}a_0(t,x)$

% }

In this proof, we are going to employ $a_k(t,x), 1\le k\le N$ obtained above, whose construction guarantees that \eqref{Family} is satisfied. First, we observe that thanks to \eqref{a_0 full equation}, $v(t,x)$, defined in \eqref{th-st-v}, can be written as follows
\begin{align}\label{v 3.3}
    v(t,x) = h^{-1/2} e^{i(x\cdot \xi/h - t|\xi|^2/h^2)} a_0(t,x,h).
\end{align}

% Then, the approximate solution 

Let us define 
\begin{align}\label{u_N}
u_N(t,x)= h^{-1/2} e^{i\frac{x\cdot \xi}{h} - i\frac{t|\xi|^2}{h^2}}(a_0+ \cdots + a_N h^N).
\end{align}
Then, by the construction of $a_k(t,x)$ and \eqref{Family}, we have  
$$i\partial_t u_N + \Delta u_N - \alpha(x) u_N|u_N|^2 = e^{i(x\cdot \xi/h -t|\xi|^2/h^2)} \left(h^{N-1/2} \Delta a_N(t,x,h) + \sum_{N+1\le k\le 3N} h^{k-3/2} g_k\right) \equiv R_N.$$
Additionally, note that
\begin{align}
    u_N(-Th,x) = h^{-1/2} e^{i\frac{x\cdot \xi}{h} + i\frac{T|\xi|^2}{h}} \psi(x).
\end{align}
The well-posedness theory for $a_k(t,x,h), N\le k\le 3N$, which follows from Proposition 3.4 and Remark 3.6, yields that $R_N= O_N(h^{N-1/2-s})$ in the $L^\infty_t H_x^{s}$ norm, for $s>d/2$. Thus, if $s=d/2+\delta(d)$ with $d=1,2,3$ and $\delta(d)$ defined in the statement of Lemma \ref{Lemma for a_0}, then $R_N= O(h^{N-1/2-3/2-\delta})= O(h^{N-2-\delta})$ for each $d$. \par
Now, let $u$ be the unique $C_t^0 H^{d/2+\delta}_x([-Th,Th]\times \R^d)$ strong solution to the initial value problem 
\begin{align}
    \begin{cases} 
   i\partial_t u + \Delta u=\alpha(x) u|u|^2 \\
    u(-Th,x) = h^{-1/2} e^{i\frac{x\cdot \xi}{h} + i\frac{T|\xi|^2}{h}} \psi(x),
   \end{cases}
\end{align}
which is guaranteed to exist by Proposition \ref{GWP H^{d/2+}}.

By the stability result, Proposition \ref{Stability}, we obtain the following estimate for the norm of the difference $u-u_N$:
\begin{align*}
    \norm{u-u_N}_{C_t^0 H^{d/2+\delta}_x([-Th, Th] \times \R^d)} \lesssim_{d, \delta, T, N} h^{N-1-\delta}.
\end{align*}
For the sake of the recovery process outlined in Section \ref{Recovery}, we take $N=3$ (larger $N$ may also be taken but it would not improve the error) so that for $0< \delta(d)<1$ 
\begin{align*}
        \norm{u-u_3}_{C_t^0 H^{d/2+\delta}_x([-Th, Th] \times \R^d)} \lesssim_{d, \delta, T} h^{2-\delta} \le h
\end{align*}
with the last inequality holding for small enough $h$. By the Sobolev embedding, we thus have
\begin{align}\label{Stability in C_{t,x}}
    \norm{u-u_3}_{C_{t,x}^0([-Th, Th] \times \R^d)} \lesssim_{\delta, d, T} h .
\end{align}
On the other hand, by \eqref{v 3.3}, \eqref{u_N} and the Sobolev embedding $ H^{d/2+\delta}_x\hookrightarrow C^0_x$, and our bounds on $\norm{a_k}_{C_t^0 H^{d/2+\delta}_x([-Th, Th] \times \R^d)}$ together with the fact that $0<h<1$,
\begin{align}\label{u_3-v}
    \norm{u_3-v}_{C_{t,x}^0 ([-Th, Th] \times \R^d)} \lesssim_{d,\delta, T} h^{1/2}.
\end{align}
Now, we can combine \eqref{u_3-v} and \eqref{Stability in C_{t,x}}, keeping in mind that $0<h<1$, to obtain by the triangle inequality, 
\begin{align}\label{u-v C^0}
    \norm{u-v}_{C_{t,x}^0 ([-Th, Th]\times\R^d)} \lesssim_{T,d, \delta} h^{1/2}.
\end{align}

We can restrict to the space-time point $\{(t,x): t=hT, x=4\xi T+x_0\}$. 
% Note that 
% \begin{align}
%     \int_{-T/h}^{T/h} \alpha(x-2\xi T/h +2\xi s)\,ds = \frac{1}{h} \int_{-T}^T \alpha(x-2\xi T/h+ 2 \xi s/h)\,ds=\frac{1}{h} \int_{-T}^T \alpha(x-2\xi T/h+ 2 \xi s/h)\,ds=: \frac{1}{h} X\alpha(x, \xi).
% \end{align}
By \eqref{u-v C^0}, we have
\begin{align} \label{u evaluated at point}
    u(Th, 4\xi T+x_0) = h^{-1/2} e^{i(3T+x_0\cdot \xi/|\xi|^2) |\xi|^2/h } e^{-i\psi(x_0)^2 X\alpha } \psi(x_0) + O(h^{1/2})
\end{align}
with
\begin{align*}
    X\alpha(x_0,\xi) := \int_{-T}^T \alpha(x_0 + 2\xi (s+T))\,ds = \int_{0}^{2T} \alpha(x_0 +2\xi s)\,ds = \frac{1}{2}\int_{0}^{4T} \alpha(x_0 +\xi s)\,ds.
\end{align*}
From now on, we let $|\xi|=1$. Now, note that if $s>|x_0|+ T_0$, then
\begin{align*}
    |x_0+s\xi| &\ge |s-|x_0|| \\
     &=s- |x_0| \\
    &> T_0,
\end{align*}
and thus $\alpha(x_0+s\xi)=0$ by using the fact that the support of $\alpha$ lies inside $B(0,T_0)$. On the other hand, using the fact that $x_0\in B(0,T_0)$, we have 
\begin{align}\label{T>T_0/2}
4T> 2T_0= T_0+T_0 >|x_0|+T_0.
\end{align}
Now, we can rewrite $X\alpha(x_0, \xi)$ as follows
\begin{align} 
    X\alpha(x_0,\xi) &= \frac{1}{2}\int_{0}^{4T} \alpha(x_0 +\xi s)\,ds \\
    &=  \frac{1}{2}\int_{0}^{\min(4T,|x_0|+T_0)} \alpha(x_0 +\xi s)\,ds \label{use alpha} \\
    &= \frac{1}{2}\int_{0}^{|x_0|+T_0} \alpha(x_0 +\xi s)\,ds \label{use min}\\
    &= \frac{1}{2}\int_{0}^{\infty} \alpha(x_0 +\xi s)\,ds \label{supp alpha},
\end{align}
where to obtain \eqref{use alpha}, we used the fact that $\alpha(x_0+s\xi)$ vanishes for $s>|x_0|+T_0$; to obtain \eqref{use min}, we used \eqref{T>T_0/2}; to obtain \eqref{supp alpha} we used again the fact that $\alpha(x_0+s\xi)$ vanishes for $s>|x_0|+T_0$. 

% From our choice of $h$; to obtain  we have that
% \begin{align}
%     \frac{3T+x_0\cdot \xi}{h}\in 2\pi \N 
% \end{align}
% Therefore
% \begin{align}
%     u(Th, 4\xi T+x_0) = h^{-1/2}  e^{-i\psi(x_0)^2 X\alpha(x_0, \xi) } \psi(x_0) + O(h^{1/2})
% \end{align}

\section{Recovery}\label{Recovery}
\subsection{Proof of Theorem \ref{Recovery Theorem}}

% Note that we can always guarantee that such $h$ exists no matter how small $h$ must be. 

Thanks to \eqref{u evaluated}, we have that
\begin{align}
        u(Th, 4\xi T+x_0) = h^{-1/2} e^{i(3T+x_0\cdot \xi) /h } e^{-i\psi(x_0)^2 X\alpha(x_0, \xi) } \psi(x_0) + O(h^{1/2})
\end{align}
where
\begin{align}
    X\alpha(x_0, \xi) = \frac{1}{2}\int_{0}^{\infty} \alpha(x_0 +\xi s)\,ds.
\end{align}
 Hence,
\begin{align}
            \psi(x_0)^{-1}e^{-i(3T+x_0\cdot \xi) /h }h^{1/2} u(Th, 4\xi T+x_0) =   e^{-i\psi(x_0)^2 X\alpha(x_0, \xi) }  + O(h).
\end{align}
Taking the logarithm of both sides, defined on $\C\setminus (-\infty,0]$,
\begin{align} \label{Recovery with g term}
    \psi(x_0)^2 X\alpha(x_0, \xi) = -\operatorname{Im} \log \left\{\psi(x_0)^{-1}e^{-i(3T+x_0\cdot \xi) /h }h^{1/2} u(Th, 4\xi T+x_0)\right\} + 2\pi g(x_0, \xi) + O(h),
\end{align}
where $g(x_0, \xi)$ is a $\Z$-valued, a priori discontinuous function. We can now use the same approach as in \cite{S_Barreto_2022} to recover $X\alpha(x_0,\xi)$, since $\psi(x_0)$ is a non-zero constant for the values of $x_0$ we need to sample. Indeed, the argument in \cite{S_Barreto_2022} only needs to sample values of $x_0\in C_{\xi, T_0} \cup \operatorname{supp}\alpha$, where
\begin{align}
    C_{\xi, T_0} = \{x_0\in \R^d: X\alpha(x_0,\xi)=0\} \cap B(0,2T_0) 
\end{align}
and in this region $\psi(x_0)=\psi(0)=K$. Note that $C_{\xi, T_0}$ is necessarily non-emtpy as it contains for example the ball $B(3T_0\xi/2, T_0/2)$. The heart of the argument of \cite{S_Barreto_2022} is based on the observation that for their analogue of $x_0\in C_{\xi, T_0}$, the left hand side of \eqref{Recovery with g term} being zero gives that the right hand side must be zero up to $O(h)$. This, in turn, implies that the discontinuities of the first two terms on the right hand side of \eqref{Recovery with g term} must cancel each other. For more details, see Section 3.1 of \cite{S_Barreto_2022}.

\appendix 
\section{}

\label{Appendix}
\subsection{Definitions}\label{Definitions}

For $I$ an interval, we define  $S^0_{t,x}(I \times \R^d)$ as the closure of Schwartz functions under the norm
\begin{align*}
    \norm{f}_{S^0_{t,x}(I\times \R^d)} \equiv \sup_{(p,q) \operatorname{admissible}} \norm{f}_{L^p_t L^q_x(I \times \R^d)}
\end{align*}
with $(p,q)$ Strichartz-admissible if, and only if $2\le p, q\le \infty$, $(p,q,d) \neq (2, \infty, 2$), and 
\begin{align*}
    \frac{2}{p}+\frac{d}{q}=\frac{d}{2}.
\end{align*}
Set 
\begin{align*}
    \norm{f}_{S^1_{t,x}(I\times \R^d)} \equiv \norm{f}_{S^0_{t,x}(I\times \R^d)} +  \norm{\nabla f}_{S^0_{t,x}(I\times \R^d)}
\end{align*}
and similarly define $S^1_{t,x}(I\times \R^d)$.
\begin{lem}\label{Algebra property}
    If $f, g\in H^{d/2+\delta}(\R^d)$ for $\delta>0$, then the following algebra property holds:
    \begin{align*}
        \norm{fg}_{H^{d/2+\delta}(\R^d)} \le C(d, \delta) \norm{g}_{H^{d/2+\delta}(\R^d)}\norm{g}_{H^{d/2+\delta}(\R^d)}.
    \end{align*}
\end{lem}
\subsection{Proof of Proposition 2.3 and 2.4}
\label{Appendix Prop 2.3 and 2.4 proof}
The local well-posedness of the $\alpha$-NLS \eqref{alpha NLS 2} in $d\in\{1,2,3\}$ dimensions follows from a Banach fixed-point contraction argument showing the existence of a fixed point of the map
\begin{align*}
    u\mapsto e^{i(t+Th)\Delta }\phi(x)+ \int_{-Th}^t e^{i(t-s)\Delta} (\alpha(x) u|u|^2(s))\,ds
\end{align*}
in the space $X=C_t^0H^1_x(I\times \R^d) \bigcap S^1_{t,x}(I \times \R^d)$ for some small interval $I$ containing $-Th$
using the fact that $\alpha, \nabla\alpha \in L^\infty_x(\R^d)$. Then, utilizing the defocusing condition $\alpha \ge 0$ and the conserved energy
\begin{align*}
    E(u(t))=\frac{1}{2} \int_{\R^d} |\nabla u|^2\,dx + \frac{1}{4} \int_{\R^d} \alpha(x) |u|^4\,dx= E(u(0)), 
\end{align*}
we can upgrade the local solution to a global solution. 
We thus have that our solution $u$ satisfies a global spacetime bound, by iteration of the local theory, 
\begin{align*}
    \norm{u}_{S^1([0, \infty)\times \R^d)} \lesssim_{} \norm{\phi}_{H^1_x(\R^d)}.
\end{align*}
% and note that this implies that in $d=3$,
% \begin{align}
%     \norm{ u}_{L^2_t L^6_x([0, \infty)\times \R^d)} + \norm{ \nabla u}_{L^2_t L^6_x([0, \infty)\times \R^d)} \lesssim \norm{\phi}_{H^1_x(\R^d)}
% \end{align}
Next, we discuss global persistence of regularity. Now, let $\phi \in H_x^{d/2+\delta(d)}(\R^d)$ with $\delta(d)>0$ such that we have that $d/2+\delta(d) \ge 1$; this is guaranteed by the choice of $\delta(d)$ in the statement of Theorem \ref{Main Theorem}. Hence, $\phi\in H^1(\R^d)$. There exists a solution $u(t,x)$ in $ C^0_t H_x^{d/2+\delta}(I\times \R^d)$ for some interval $I$ which satisfies 
\begin{align*}
    u=e^{i(t+Th)\Delta }\phi(x)+ \int_{-Th}^t e^{i(t-s)\Delta} (\alpha(x) u|u|^2(s))\,ds.
\end{align*}
Furthermore, because $\phi \in H^1(\R^d)$, we have that $\norm{ u}_{S^1([0, \infty)\times \R^d)} \lesssim \norm{\phi}_{H^1_x}$. By Minkowski, and the estimate
\begin{align*}
    \norm{fg}_{H^{d/2+\delta}_x(\R^d)} \lesssim_{d,\delta} \norm{f}_{L^\infty(\R^d)} \norm{g}_{H^{d/2+\delta}_x(\R^d)} + \norm{g}_{L^\infty(\R^d)} \norm{f}_{H^{d/2+\delta}_x(\R^d)},
\end{align*}
we have 
\begin{align*}
    \norm{u(t)}_{H^{d/2+\delta}_x(\R^d)} \le \norm{\phi}_{H^{d/2+\delta}_x(\R^d)} + C_{d,\delta} \int_{-Th}^t \norm{\alpha}_{H^{d/2+\delta}_x(\R^d)} \norm{u(t)}_{L^\infty(\R^d)}^2 \norm{u(t)}_{H^{d/2+\delta}_x(\R^d)}\,ds
\end{align*}
which leads to, by Gronwall,  
\begin{align*}
    \norm{u}_{C^0_tH^{d/2+\delta}_x(I\times \R^d)} &\le \norm{\phi }_{H^{d/2+\delta}_x(\R^d)} \exp\left(C_{d,\delta} \norm{\alpha}_{H^{d/2+\delta}_x(\R^d)}\norm{u}_{L^2_t L^\infty_x(I\times \R^d)}^2\right) \\
    &\le \norm{\phi}_{H^{d/2+\delta}_x(\R^d)} \exp\left(C_{d,\delta} \norm{\alpha}_{H^{d/2+\delta}_x(\R^d)}\norm{u}_{L^2_t L^\infty_x(I\times \R^d)}^2\right).
\end{align*}
This is where we need to treat the estimates differently depending on the dimension.
\par 
If $d=3$, we use Hölder to access the Strichartz pair $\left(\frac{12+4\e}{3+3\e}, 3+\e\right)$ and the Sobolev embedding $W^{1, 3+\e}(\R^3)\hookrightarrow L^\infty(\R^3)$ for $0<\e \ll 1$ to estimate 
\begin{align*}
    \norm{u}_{L^2_t L^\infty_x(I\times \R^3)}^2 &\le |I|^{(6-2\e)/(12+4\e)} \norm{u}^2_{L^{(12+4\e)/(3+3\e)}_t L^\infty_x(I\times \R^3)} \\
    &\lesssim |I|^{1/2+O(\e)} \norm{\langle \nabla \rangle u}^2_{L^{(12+4\e)/(3+3\e)}_t L^{3+\e}_x(I\times \R^3)} \\
    &\lesssim |I|^{1/2+O(\e)} \norm{\phi}_{H^1(\R^3)}^2
\end{align*}
and thus 
% have that by Sobolev, the inequality $|\xi|^{1/2}\lesssim 1+|\xi|$, and the boundedness of the Fourier multiplier operator $\frac{|\nabla|^{1/2}}{1+|\nabla|}:L^6_x\to L^6_x$ (with operator norm bounded by $C'$)
\begin{align*}
    \norm{u}_{C^0_tH^{3/2+\delta}_x(I\times \R^3} &\le \norm{\phi }_{H^{3/2+\delta}_x(\R^3)} \exp\left(C_{3,\delta} \norm{\alpha}_{H^{3/2+\delta}_x(\R^3)} |I|^{1/2+O(\e)} \norm{\phi}_{H^1(\R^3)}^2\right)
\end{align*}
for any interval $I$. This allows us to iterate the local solution to an arbitrarily large interval and the solution is thus global. \par 
If $d=2$, notice that $(2, \infty)$ is is the endpoint estimates which fails, and thus we will need to apply Hölder, followed by the Sobolev embedding $ W^{1, 2+\e}(\R^2) \hookrightarrow L^\infty(\R^2)$ for $0<\e \ll 1$ to access a Strichartz norm $(\frac{2(2+\e)}{\e}, 2+\e)$, as follows
% (We now take $I=[-Th,-Th+S]$),
\begin{align*}
    \norm{u}_{L^2_t L^\infty_x(I\times \R^2)}^2 
    &\le |I|^{1/(1+\e/2)}  \norm{ u}_{L^{2(2+\e)/\e}_t L^{\infty}_x(I\times \R^2)}^2 \\
    &\lesssim |I|^{1+O(\e)} \norm{\langle \nabla \rangle u}_{L^{2(2+\e)/\e}_t  L^{2+\e }_x(I\times \R^2)}^2 \\
    &\lesssim |I|^{1+O(\e)} \norm{\phi}_{H^1_x(\R^d)}^2
\end{align*}
and thus we have
\begin{align*}
    \norm{u}_{C^0_tH^{d/2+\delta}_x(I\times \R^2)} \le \norm{\phi }_{H^{3/2+\delta}_x(\R^3)} \exp\left(C_{ \delta} \norm{\alpha}_{H^{1+\delta}_x(\R^2)} |I|^{1+O(\e)}\norm{\phi }_{H^{1}_x(\R^3)}^2\right)
\end{align*}
immediately giving global existence by iteration of the local theory. \par
For $d=1$, using the embedding $H_x^{1/2+\e} \hookrightarrow L_x^\infty$ for $0<\e\ll 1$, we have that
\begin{align}
    \norm{u}_{L^2_t L^\infty_x(I\times \R^2)}^2 &\le |I|  \norm{u}_{L^\infty_t L^\infty_x(I\times \R^2)}^2 \\
    &\lesssim |I|  \norm{\langle \nabla \rangle^{1/2+\e} u}_{L^\infty_t L^2_x(I\times \R^2)}^2 \\
    &\lesssim |I|  \norm{\langle \nabla \rangle u}_{L^\infty_t L^2_x(I\times \R^2)}^2 \\
    &\lesssim |I| \norm{\phi}_{H^1}^2,
\end{align}
which again gives global existence by the above bounds.
\subsection{Proof of Proposition 2.6 }
\label{Appendix Prop 2.6 proof}
\begin{proof} 
    By Duhamel, we have that
    \begin{align}
        u(t,x)-u_N(t,x)= \int_{-Th}^t e^{i(t-s)\Delta} \alpha(x)\left(u|u|^2 -  u_N|u_N|^2\right)\,ds + \int_{-Th}^t e^{i(t-s)\Delta}\alpha(x) R_N\,ds.
    \end{align}

Let
\begin{align*}
    C(T, h, N, \alpha, \psi, d, \delta )=\max\left(\norm{u}_{C_t^0 H^{d/2+\delta }_x([-Th, Th)\times \R^d)},\norm{u_N}_{C_t^0 H^{d/2+\delta }_x([-Th, Th)\times \R^d)}\right).
\end{align*}
Using the identity
\begin{align*}
    u|u|^2 -  u_N|u_N|^2 = (u-u_N)\left(u_N\overline{u_N} + \overline{u}u_N\right) + \overline{(u-u_N)} u^2
\end{align*}
and the algebra property of $H^{d/2+\delta}_x(\R^d)$ with the triangle inequality and the fact that $t\in[-Th, Th]$,
\begin{align*}
    \norm{u(t,x)-u_N(t,x)}_{H^{d/2+\delta}_x(\R^d)} &\le \int_{-Th}^t \norm{\alpha}_{H^{d/2+\delta}_x} 3 C(d, \delta)^3 C(T, h, N, \alpha, \psi, d, \delta )^2 \norm{u(x,s)-u_N(x,s)}_{H^{d/2+\delta}_x(\R^d)}\,ds \\
    &+ 2hT \, \norm{\alpha}_{H^{d/2+\delta}_x}C(d, \delta) \,\norm{R_N}_{C_t^0 H^{d/2+\delta}_x([-Th, Th]\times \R^d)}.
\end{align*}
Thus, by Grönwall's inequality, 
\begin{align*}
    \norm{u(t,x)-u_N(t,x)}_{H^{d/2+\delta}_x(\R^d)} &\le 2hT \, \norm{\alpha}_{H^{d/2+\delta}_x}C(d, \delta) \exp\left(\int_{-Th}^t C(d, \delta)^3 \norm{\alpha}_{H^{d/2+\delta}_x}C(T, h, N, \alpha, \psi, d, \delta )^2\,ds \right) \\
    &\times \norm{R_N}_{C_t^0 H^{d/2+\delta}_x([-Th, Th]\times \R^d)} \\
    &\le 2hT \norm{\alpha}_{H^{d/2+\delta}_x}\, C(d, \delta) \exp\left(2hT\, C(d, \delta)^3 \norm{\alpha}_{H^{d/2+\delta}_x}C(T, h, N, \alpha, \psi, d, \delta )^2 \right) \\
    &\times \norm{R_N}_{C_t^0 H^{d/2+\delta}_x([-Th, Th]\times \R^d)} \\
    &\lesssim_{d, \delta, T, N, \alpha, \psi } h \norm{R_N}_{C_t^0 H^{d/2+\delta}_x([-Th, Th]\times \R^d)}
\end{align*}
and the result follows by taking the supremum in time on both sides. 
\end{proof}
\nocite{*} 
\bibliography{References}
\bibliographystyle{plain}

\end{document}